\pdfoutput=1 
\documentclass[11pt]{amsart}
\usepackage[utf8]{inputenc}
\usepackage{amssymb,amsmath,enumerate,amsthm,enumitem,colonequals,mlmodern, tikz-cd, microtype}
\usepackage{microtype}
\usepackage[mathcal]{eucal}
\usepackage[top=3.75cm, bottom=3cm, left=3.5cm, right=3.5cm]{geometry}

\makeatletter
\@namedef{subjclassname@2020}{\textup{2020} Mathematics Subject Classification}
\makeatother

\usepackage{xcolor}
\colorlet{darkblue}{blue!55!black}
\colorlet{darkcyan}{cyan!50!black}
\colorlet{darkgreen}{green!60!black}

\usepackage{hyperref}
\hypersetup{
    colorlinks=true,
    linkcolor= darkblue,
    urlcolor= darkcyan,
    citecolor= darkgreen,
}

\def\eqref#1{\textcolor{darkblue}{(\ref{#1})}}

\PassOptionsToPackage{hyphens}{url}\usepackage{hyperref}

\usepackage[nameinlink]{cleveref} 
\Crefformat{section}{#2\S#1#3} 
\Crefmultiformat{section}{#2\S\S#1#3}{ and~#2#1#3}{, #2#1#3}{, and~#2#1#3}

\usepackage[pagewise]{lineno}
\let\oldequation\equation
\let\oldendequation\endequation
\renewenvironment{equation}{\linenomathNonumbers\oldequation}{\oldendequation\endlinenomath}
\expandafter\let\expandafter\oldequationstar\csname equation*\endcsname
\expandafter\let\expandafter\oldendequationstar\csname endequation*\endcsname
\renewenvironment{equation*}{\linenomathNonumbers\oldequationstar}{\oldendequationstar\endlinenomath}
\let\oldalign\align
\let\oldendalign\endalign
\renewenvironment{align}{\linenomathNonumbers\oldalign}{\oldendalign\endlinenomath}
\expandafter\let\expandafter\oldalignstar\csname align*\endcsname
\expandafter\let\expandafter\oldendalignstar\csname endalign*\endcsname
\renewenvironment{align*}{\linenomathNonumbers\oldalignstar}{\oldendalignstar\endlinenomath}

\usepackage{tikz-cd}




\newcounter{intro}
\newcounter{result}

\newtheorem{introthm}[intro]{Theorem}

\theoremstyle{definition}
\newtheorem{theorem}[result]{Theorem}
\newtheorem{lemma}[result]{Lemma}
\newtheorem{proposition}[result]{Proposition}
\newtheorem{corollary}[result]{Corollary}
\newtheorem{definition}[result]{Definition}
\newtheorem{example}[result]{Example}
\newtheorem{remark}[result]{Remark}

\newtheorem{setup}[result]{Setup}

\newtheorem*{ack}{Acknowledgements}

\numberwithin{equation}{section}
\numberwithin{result}{section}

\title[Triangulated characterizations of singularities]{Triangulated characterizations of singularities}

\author[P.~Lank]{Pat Lank}
\address{P.~Lank,
Dipartimento di Matematica “F. Enriques”, Universit\`{a} degli Studi di Milano, Via Cesare
Saldini 50, 20133 Milano, Italy}
\email{plankmathematics@gmail.com}

\author[S.~Venkatesh]{Sridhar Venkatesh}
\address{S.~Venkatesh,
Department of Mathematics,
University of Michigan, 
Ann Arbor, MI 48109,
U.S.A.}
\email{srivenk@umich.edu}

\date{\today}



\begin{document}
    
\begin{abstract}
    This work presents a range of triangulated characterizations for important classes of singularities such as derived splinters, rational singularities, and Du Bois singularities. An invariant called 'level' in a triangulated category can be used to measure the failure of a variety to have a prescribed singularity type. We provide explicit computations of this invariant for reduced Nagata schemes of Krull dimension one and for affine cones over smooth projective hypersurfaces. Furthermore, these computations are utilized to produce upper bounds for Rouquier dimension on the respective bounded derived categories.
\end{abstract}

\maketitle

\section{Introduction}
\label{sec:intro}
 
This article presents a technical result concerning the splitting of the natural map $\mathcal{O}_X \to \mathbb{R}\pi_\ast \mathcal{O}_Y$ in the bounded derived category of coherent sheaves on $X$, for a given proper surjective morphism $\pi\colon Y \to X$ of Noetherian schemes (see Lemma~\ref{lem:splitting_via_thick_tricat}). Our result offers a new perspective and set of tools for detecting various important classes of singularities studied in algebraic geometry.

We employ a concept called generation in a triangulated category $\mathcal{T}$, which has been extensively studied \cite{BVdB:2003, ABIM:2010, Rouquier:2008}. Let $E$ be an object and $\mathcal{C}$ a subcategory of $\mathcal{T}$. We say that $\mathcal{C}$ \textit{finitely builds} $E$ if $E$ can be obtained from objects in $\mathcal{C}$ using a finite combination of coproducts, shifts, direct summands, and cones. The minimal number of cones required is the \textit{level} of $E$ with respect to $\mathcal{C}$, denoted $\operatorname{level}^{\mathcal{C}}(E)$. A special case of interest is the case where $\mathcal{C}$ consists of a single object. The \textit{Rouquier dimension} of $\mathcal{T}$ is the smallest integer $n$ such that there exists an object $G$ where $\operatorname{level}^G (E) \leq n+1$ for all objects $E$. See Section~\ref{sec:generation} for details.

Let $D^b_{\operatorname{coh}}(X)$ denote the bounded derived category of coherent sheaves on a Noetherian scheme $X$. For a proper surjective morphism $\pi\colon Y\to X$, the level of $\mathcal{O}_X$ with respect to $\mathbb{R}\pi_\ast D^b_{\operatorname{coh}}(Y)$ (i.e. the essential image of $\mathbb{R}\pi_\ast \colon D^b_{\operatorname{coh}}(Y) \to D^b_{\operatorname{coh}}(X)$) is always finite, see \cite[Theorem 1.1]{Dey/Lank:2024}. We will show that, under mild constraints, the splitting of the natural map $\mathcal{O}_X \to \mathbb{R}\pi_\ast \mathcal{O}_Y$ is equivalent to the level of $\mathcal{O}_X$ with respect to $\mathbb{R}\pi_\ast D^b_{\operatorname{coh}}(Y)$ being at most one in $D^b_{\operatorname{coh}}(X)$. The splitting of this natural map, in the appropriate settings, is used to define various classes of singularities. 

To set the stage, let us consider singularities in prime characteristic. Suppose $X$ is a variety over a perfect field of prime characteristic, and denote its Frobenius morphism by $F \colon X \to X$. We say that $X$ is \textit{globally $F$-split} if the natural map $F \colon X \to X$ splits as $\mathcal{O}_X$-modules, see \cite{Mehta/Ramanathan:1985}. In \cite{BILMP:2023}, techniques of generation were utilized to prove that when $X$ is affine, it is globally $F$-split if, and only if, the level of $\mathcal{O}_X$ with respect to $F_\ast \mathcal{O}_X$ is at most one. This alludes one to characterizing a prescribed singularity type using generation. We push this line of thought further by studying classes of singularities in various contexts through the lens of generation in a triangulated category.

Let us shift gears by highlighting classes of singularities that are of interest to our work; for background, see Section~\ref{sec:singularities}:
\begin{itemize}
    \item A variety $X$ over a field of characteristic zero has \textit{rational singularities} if, given a resolution of singularities $\pi\colon \widetilde{X}\to X$, the natural map $\mathcal{O}_X \to \mathbb{R}\pi_\ast \mathcal{O}_{\widetilde{X}}$ splits in $D^b_{\operatorname{coh}}(X)$ \cite{Kovac:2000, Bhatt:2012}.
    \item A variety over a field is said to be a \textit{derived splinter} variety $X$ if the natural map $\mathcal{O}_X \to \mathbb{R}\pi_\ast \mathcal{O}_Y$ splits in $D^b_{\operatorname{coh}}(X)$ for all proper surjective morphisms $\pi \colon Y \to X$ \cite{Bhatt:2012}.
    \item A variety $X$ over a field of characteristic zero is said to have \textit{Du Bois singularities} if the natural map $\mathcal{O}_X \to \underline{\Omega}^0_X$ splits in $D^b_{\operatorname{coh}}(X)$ \cite{Schwede:2007, Kovacs:1999}, where $\underline{\Omega}^0_X$ denotes the $0$-th graded piece of the Du Bois complex of $X$.
\end{itemize}

These classes of singularities are characterized by a splitting condition, which naturally fits into the context of generation in $D^b_{\operatorname{coh}}(X)$ as they tell us that $\mathcal{O}_X$ is a direct summand of a particular object. We now state our first main result, which is a special case of Theorem~\ref{thm:derived_splinters_via_generation}, Theorem~\ref{thm:rational_singularities_via_generation}, Proposition~\ref{prop:rational_singularities_perfect}, and Theorem~\ref{thm:du_bois_via_generation}.

\vspace{3mm}

\begin{introthm}\label{introthm:rational_derived_splinter_db_generation}
    Let $X$ be a variety over a field $k$.
    \begin{enumerate}
        \item The following are equivalent:
        \begin{enumerate}
            \item $X$ is a derived splinter
            \item $\operatorname{level}^{\mathbb{R}\pi_\ast D^b_{\operatorname{coh}}(Y)} (\mathcal{O}_X) \leq 1$ for all proper surjective morphisms $\pi \colon Y \to X$.
        \end{enumerate}

        \item Assume $k$ has characteristic zero. If $\pi \colon \widetilde{X} \to X$ is a resolution of singularities, then the following are equivalent:
        \begin{enumerate}
            \item $X$ has rational singularities
            \item $\operatorname{level}^{\mathbb{R}\pi_\ast D^b_{\operatorname{coh}}(\widetilde{X})} (\mathcal{O}_X)\leq 1$
            \item $\mathbb{R}\pi_\ast \mathcal{O}_{\widetilde{X}}$ is a perfect complex on $X$.
        \end{enumerate}

        \item Assume $k$ has characteristic zero. If $X$ is an embeddable variety, then the following are equivalent:
        \begin{enumerate}
            \item $X$ has Du Bois singularities
            \item $\operatorname{level}^{\underline{\Omega}^0_X} (\mathcal{O}_X) \leq 1$.
        \end{enumerate}
    \end{enumerate}
\end{introthm}

Theorem~\ref{introthm:rational_derived_splinter_db_generation} is a special case of more general statements found in Section~\ref{sec:proofs}. The key technical ingredient that bridge the splitting conditions to statements regarding generation is Lemma~\ref{lem:splitting_via_thick_tricat}. In forthcoming work, we will establish similar results regarding singularities of pairs \cite{Lank/McDonald/Venkatesh:2025}. This perspective of utilizing generation in triangulated categories allows one to use level as an invariant to measure the failure of a prescribed singularity type.

Along the way, we discover that it is possible to make similar assertions about the normality of an integral Nagata scheme in this context. This is particularly relevant for varieties over a field when investigating their potential for normality.

\begin{introthm}\label{introthm:normality_via_generation}
    If $X$ is an integral Nagata scheme, then the following are equivalent:
    \begin{enumerate}
        \item $X$ is normal
        \item $\operatorname{level}^{\mathbb{R} \nu_\ast \mathcal{O}_{X^\nu}} (\mathcal{O}_X) \leq 1$ where $\nu: X^\nu \to X$ is the normalization map.
    \end{enumerate}
\end{introthm}

Furthermore, we provide explicit computations yielding bounds on the number of cones required to finitely build $\mathcal{O}_X$ in $D^b_{\operatorname{coh}}(X)$ from a given resolution of singularities $\pi \colon \widetilde{X} \to X$ for a variety $X$ over a field. The examples of interest include quasi-projective curves and affine cones over a smooth projective hypersurface over a field of characteristic zero. These calculations measure the failure to have rational singularities and provide upper bounds on Rouquier dimension for the corresponding bounded derived categories, see Proposition~\ref{prop:rouquier_dimension_curves}, Corollary~\ref{cor:affine_curve_level}, Proposition~\ref{prop:cones_inequality} and Corollary~\ref{cor:cones_rouq_dim}.

\begin{introthm}
    \begin{enumerate}
        \item  Let $X$ be a reduced Nagata one-dimensional scheme. If $\nu \colon X^\nu \to X$ is the normalization of $X$, then the Rouquier dimension of $D^b_{\operatorname{coh}}(X)$ is bounded above by the following value:
        \begin{displaymath}
            (\dim D^b_{\operatorname{coh}}(X^\nu) +1 )(1 + \underset{s \in \operatorname{Sing}(X)}{\max} \{ \delta_s\}) -1
        \end{displaymath}
        where $\delta_s$ denotes the $\delta$-invariant of $X$ at $s$. If $X$ is affine, then one has the following:
        \begin{displaymath}
            \operatorname{level}^{\mathbb{R}\nu_\ast \mathcal{O}_{X^\nu}}(\mathcal{O}_X)\leq 1 + \underset{s \in \operatorname{Sing}(X)}{\max} \{ \delta_s\}.
        \end{displaymath}
        \item Let $X$ be a smooth projective hypersurface in $\mathbb{P}^n_k$ of degree $d \geq n$, where $k$ is a field of characteristic zero. Let $C$ be the affine cone of $X$ associated to an ample line bundle on $X$. Then
        \begin{displaymath}
            \dim D^b_{\operatorname{coh}} (C) \leq (2 \dim C + 1) (1 + 2(d-n)) -1.
        \end{displaymath}
        Moreover, there is an inequality:
        \begin{displaymath}
        \operatorname{level}^{\mathbb{R}\pi_\ast(\mathcal{O}_{\widetilde{C}} \oplus \mathcal{O}_{\widetilde{C}}(E))}(\mathcal{O}_C) \leq 1+ 2(d-n).
        \end{displaymath}
    \end{enumerate}
\end{introthm}

\subsection{Notation}
\label{sec:intro_notation}

Let $X$ be a Noetherian scheme. The following will be of importance:
\begin{enumerate}
    \item $D(X):=D(\operatorname{Mod}(X))$ is the derived category of $\mathcal{O}_X$-modules.
    \item $D_{\operatorname{Qcoh}}(X)$ is the (strictly full) subcategory of $D(X)$ consisting of complexes with quasi-coherent cohomology.
    \item $D_{\operatorname{coh}}^b(X)$ is the (strictly full) subcategory of $D(X)$ consisting of complexes having bounded and coherent cohomology.
    \item $\operatorname{perf}(X)$ is the (strictly full) subcategory of $D_{\operatorname{Qcoh}}(X)$ consisting of the perfect complexes on $X$.
\end{enumerate}
If $X$ is affine, then we might at times abuse notation and write $D(R):=D_{\operatorname{Qcoh}}(X)$ where $R:=H^0(X,\mathcal{O}_X)$ are the global sections; similar conventions will occur for the other categories.

\begin{ack}
    We thank Srikanth B. Iyengar, Peter McDonald, Mircea Musta\c{t}\u{a}, Amnon Neeman, and Karl Schwede for comments and discussions on earlier versions of this work. The first author was partially supported by the National Science Foundation under Grant No. DMS-1928930 while the author was in residence at the Simons Laufer Mathematical Sciences Institute (formerly MSRI). The second author was partially supported by the National Science Foundation under Grant No. DMS-2101874. The authors thank Bhargav Bhatt for private communications regarding improvements to earlier work, particularly for Lemma~\ref{lem:splitting_via_thick_tricat} and Lemma~\ref{lem:bhatt_lemma}. Both authors would also like to thank the anonymous referee.
\end{ack}

\section{Generation}
\label{sec:generation}

This section provides a streamlined overview of generation in triangulated categories, and utilizes content found in \cite{BVdB:2003, Rouquier:2008, Neeman:2021,ABIM:2010}. Let $\mathcal{T}$ be a triangulated category with the shift functor $[1]\colon \mathcal{T} \to \mathcal{T}$.

\begin{definition}(\cite{BVdB:2003})\label{def:thick_subcategories}
    Suppose $\mathcal{S}$ is a subcategory of $\mathcal{T}$.
    \begin{enumerate} 
        \item A triangulated subcategory is said to be \textbf{thick} if it is closed under direct summands.
        \item The smallest thick subcategory containing $\mathcal{S}$ in $\mathcal{T}$ is denoted by $\langle \mathcal{S} \rangle$. If $\mathcal{C}$ consists of a single object $G$, then we write $\langle \mathcal{C}\rangle$ as $\langle G\rangle$.
        \item Consider the following additive subcategories of $\mathcal{T}$:
        \begin{enumerate}
            \item $\operatorname{add}(\mathcal{S})$ is the smallest strictly full subcategory of $\mathcal{T}$ containing $\mathcal{S}$ and is closed under shifts, finite coproducts, and direct summands
            \item $\langle \mathcal{S} \rangle_0 : = \operatorname{add}(0)$
            \item $\langle \mathcal{S} \rangle_1 := \operatorname{add}(\mathcal{S})$
            \item $\langle \mathcal{S} \rangle_n := \operatorname{add} \{ \operatorname{cone}(\phi) : \phi \in \operatorname{Hom}_{\mathcal{T}} (\langle \mathcal{S} \rangle_{n-1}, \langle \mathcal{S} \rangle_1) \}$.
        \end{enumerate}
        If $\mathcal{C}$ consists of a single object $G$, then we write $\langle \mathcal{C}\rangle_n$ as $\langle G\rangle_n$.
    \end{enumerate}
\end{definition}

\begin{remark}
    \hfill
    \begin{enumerate}
        \item There are different forms of notation for the ideas in Definition~\ref{def:thick_subcategories}, which we highlight. On one hand, at times $\langle - \rangle_n$ is denoted by $\operatorname{thick}_{\mathcal{T}}^n (-)$, e.g. in \cite{ABIM:2010}. On the other hand, $\operatorname{add}(-)$ is sometimes taken to be the smallest strictly full subcategory of $\mathcal{T}$ containing $\mathcal{S}$ and is closed under finite coproducts, see \cite[Reminder 1.1]{Neeman:2021} and its proceeding remark.
        \item If $\mathcal{S}$ is a subcategory of $\mathcal{T}$, then there exists an exhaustive filtration on the smallest thick subcategory in $\mathcal{T}$ containing $\mathcal{S}$:
        \begin{displaymath}
            \langle \mathcal{S} \rangle_0 \subseteq \langle \mathcal{S} \rangle_1 \subseteq \cdots \subseteq \bigcup^{\infty}_{n=0} \langle \mathcal{S} \rangle_n = \langle \mathcal{S} \rangle.
        \end{displaymath}
    \end{enumerate}
\end{remark}

\begin{definition}\label{def:strong_generators}(\cite{ABIM:2010}, \cite{Rouquier:2008})
    Let $E,G$ be objects and $\mathcal{C}$ a subcategory of $\mathcal{T}$.
    \begin{enumerate}
        \item An object $E$ is \textbf{finitely built} by $\mathcal{C}$ if $E$ belongs to $\langle \mathcal{C} \rangle$. The \textbf{level}, denoted $\operatorname{level}_\mathcal{T}^{\mathcal{C}} (E)$, of $E$ with respect to $\mathcal{C}$ is the minimal non-negative integer $n$ such that $E$ is in $\langle \mathcal{C} \rangle_n$. If $\mathcal{C}$ consists of a single object $G$, then we write $\operatorname{level}^{\mathcal{C}} (E)$ as $\operatorname{level}^G (E)$.
        \item A \textbf{classical generator} for $\mathcal{T}$ is an object $G$ in $\mathcal{T}$ satisfying $\langle G \rangle =
        \mathcal{T}$.
        \item A \textbf{strong generator} for $\mathcal{T}$ is an object $G$ satisfying $\langle G \rangle_n =
        \mathcal{T}$ for some $n\geq 0$. The \textbf{generation time} of $G$ is the minimal $n$ such that for all $E$ in $ \mathcal{T}$ one has $\operatorname{level}_\mathcal{T}^G (E)\leq n+1$. 
        \item The \textbf{Rouquier dimension} of $\mathcal{T}$, denoted $\dim \mathcal{T}$, is the smallest integer $d$ such that $\langle G \rangle_{d+1} = \mathcal{T}$ for some object $G$ in $ \mathcal{T}$.
    \end{enumerate}
\end{definition}

\begin{example}\label{ex:strong_generator_examples}
    \hfill
    \begin{enumerate}
        \item If $X$ is a quasi-affine Noetherian regular scheme of (finite) Krull dimension $n$, then $\langle \mathcal{O}_X \rangle_{n+1} = D^b_{\operatorname{coh}}(X)$, see \cite[Corollary 5]{Olander:2023}.
        \item If $X$ is a smooth quasi-projective variety $X$ over a field with a very ample line bundle $\mathcal{L}$, then $\langle \bigoplus^{\dim X}_{i=0} \mathcal{L}^{\otimes i} \rangle_{2 \dim X  +1 } = D^b_{\operatorname{coh}}(X)$, see \cite[Proposition 7.9]{Rouquier:2008}.
        \item If $X$ is a quasi-projective variety $X$ over a perfect field of nonzero characteristic with a very ample line bundle $\mathcal{L}$, then $F_\ast^e (\bigoplus^{\dim X}_{i=0} \mathcal{L}^{\otimes i})$ is a strong generator for $D^b_{\operatorname{coh}}(X)$ where $e \gg 0$ and $F \colon X \to X$ is the Frobenius morphism, see \cite[Corollary 3.9]{BILMP:2023}.
        \item If $X$ is a separated quasi-excellent Noetherian scheme of finite Krull dimension, then $D^b_{\operatorname{coh}}(X)$ admits a strong generator, see \cite[Main Theorem]{Aoki:2021}. For related work in the affine setting, see \cite{Iyengar/Takahashi:2019, Dey/Lank/Takahashi:2023}.
        \item If $X$ is a Noetherian scheme such that every closed integral subscheme of $X$ has an open regular locus, then $D^b_{\operatorname{coh}}(X)$ admits a classical generator, see \cite[Theorem 1.1]{Dey/Lank:2024a}. See \cite{Clark/Lank/Parker/ManaliRahul:2024,Neeman:2022} for further connections to regular locus to other notions besides `generation'.
        \item There has been recent attention towards studying generation in \textit{noncommutative} settings. See \cite{DeDeyn/Lank/ManaliRahul:2024a,DeDeyn/Lank/ManaliRahul:2024b} for generalizations of the work above to the case of coherent alegbras over a Noetherian scheme. This largely enhances the previously known setting of \cite{Elagin/Lunts/Schnurer:2020, Bhaduri/Dey/Lank:2023}.
    \end{enumerate}
\end{example}

\begin{definition}(\cite{Atiyah:1956, Walker/Warfield:1976})
    Let $\mathcal{C}$ be an additive category. 
    \begin{enumerate}
        \item $\mathcal{C}$ is said to be a \textbf{Krull-Schmidt} category if every object in $\mathcal{C}$ decomposes into a finite coproduct of objects having local endomorphism rings.
        \item An object of $\mathcal{C}$ is called \textbf{indecomposable} if it is not isomorphic to a direct sum of two nonzero objects. 
    \end{enumerate}
\end{definition}

\begin{remark}\label{rmk:krull_schmidt_categories}
    \hfill
    \begin{enumerate}
        \item If $\mathcal{C}$ is a Krull-Schmidt category, then every object decomposes into a finite coproduct of indecomposables, and this is unique up to permutations, see \cite[Theorem 4.2]{Krause:2015}.
        \item Let $R$ be a complete Noetherian local ring. Consider an $R$-linear abelian category $\mathcal{A}$.
        \begin{enumerate}
            \item $\mathcal{A}$ is said to be \textit{$\operatorname{Ext}$-finite} if the $R$-module $\operatorname{Ext}^n (A,B)$ is finitely generated for all objects $A,B$ in $\mathcal{A}$ and $n\geq 0$.
            \item $\mathcal{A}$ is $\operatorname{Ext}$-finite over $R$ if, and only if, $D^b (\mathcal{A})$ is $\operatorname{Hom}$-finite over $R$.
            \item If $\mathcal{A}$ is $\operatorname{Ext}$-finite (over $R$), then $D^b (\mathcal{A})$ is a Krull-Schmidt category, see \cite[Corollary B]{Le/Chen:2007}.
        \end{enumerate}
    \end{enumerate}
\end{remark}

The following was originally observed for the category of coherent sheaves on a proper variety over an algebraically closed field \cite[Theorem 2]{Atiyah:1956}.

\begin{lemma}(Atiyah)\label{lem:projective_over_complete_local_ring_krs_category}
    If $X$ is a proper scheme over a Noetherian complete local ring $R$, then both $\operatorname{coh}X$ and $D^b_{\operatorname{coh}}(X)$ are Krull-Schmidt categories. Additionally, if $X$ is integral, then $\mathcal{O}_X$ is an indecomposable object in both categories.
\end{lemma}

\begin{proof}
    If $X$ is proper over $R$, then $\operatorname{coh}X$ is $\operatorname{Ext}$-finite (over $R$). This can be deduced from \cite[\href{https://stacks.math.columbia.edu/tag/0D0T}{Tag 0D0T}]{StacksProject}. Hence, we see that $D^b_{\operatorname{coh}}(X)$ is a Krull-Schmidt category, see Remark~\ref{rmk:krull_schmidt_categories}. As $X$ is proper over a Noetherian complete local ring, $H^0(X,\mathcal{O}_X)$ is a finite product of complete local rings. But $X$ being integral ensures that $H^0 (X, \mathcal{O}_X)$ is integral domain, so $H^0(X,\mathcal{O}_X)$ must be a Noetherian complete local integral domain. This tells us that $\operatorname{Ext}^0 (\mathcal{O}_X, \mathcal{O}_X)$ is a local ring. Hence, $\mathcal{O}_X$ is indecomposable via Lemma 5.2 and Proposition 5.4 of \cite{Krause:2015}.
\end{proof}

\section{Characterizations}
\label{sec:characterizations}

The primary objective of this section is to provide a triangulated
characterization of splinters, derived splinters, rational singularities, and Du
Bois singularities. After doing so, we study, for a given proper surjective morphism $\pi \colon Y \to X$, what information can be extracted by number of cones needed for $\mathcal{O}_X$ and $\mathbb{R}\pi_\ast \mathcal{O}_Y$ to build one another in $D^b_{\operatorname{coh}}(X)$. 

\subsection{Singularities}
\label{sec:singularities}

Before diving ahead into the proofs, we detour through a quick reminder on the various notions of singularities relevant to our work.

\begin{definition}(\cite{Sing:1999, Ma:1988})
    A Noetherian scheme $X$ is said to be a \textbf{splinter} if the natural map $\mathcal{O}_X \to \pi_\ast \mathcal{O}_Y$ splits in $\operatorname{coh} X$ for all finite surjective morphisms $\pi \colon Y \to X$.
\end{definition}

\begin{example}
    A connected Noetherian $\mathbb{Q}$-scheme is a splinter if, and only if, it is normal. This is \cite[Example 2.1]{Bhatt:2012}. For instance, any normal quasi-projective variety over $\mathbb{C}$ is a splinter.
\end{example}

\begin{definition}(\cite{Bhatt:2012})
    A Noetherian scheme $X$ is said to be a \textbf{derived splinter} if the natural map $\mathcal{O}_X \to \mathbb{R}\pi_\ast \mathcal{O}_Y$ splits in $D^b_{\operatorname{coh}}(X)$ for all proper surjective morphisms $\pi \colon Y \to X$.
\end{definition}

\begin{example}
    \hfill
    \begin{enumerate}
        \item A Noetherian $\mathbb{F}_p$-scheme is a derived splinter if, and only if, it is a splinter. This is \cite[Theorem 1.4]{Bhatt:2012}.
        \item Any derived splinter must be a splinter.
    \end{enumerate}
\end{example}

\begin{definition}(\cite{Kempf/Knudsen/Mumford/Saint-Donat:1973, Kollar:2013})
    \begin{enumerate}
        \item Let $(R,\mathfrak{m})$ be a quasi-excellent Noetherian local ring of equal characteristic zero. We say $R$ has \textbf{rational singularities} if $R$ is normal and if for every proper birational morphism $\pi \colon X \to \operatorname{Spec}(R)$ with $X$ regular, we have $\mathbb{R}^j \pi_\ast \mathcal{O}_X = 0$ for $j>0$.
        \item Let $Y$ be a quasi-excellent Noetherian normal scheme of equal characteristic zero. We say $Y$ has \textbf{rational singularities} if every local ring $\mathcal{O}_{Y,p}$ has rational singularities for all $p$ in $Y$.
    \end{enumerate}
\end{definition}

\begin{remark}\label{rmk:rational_singularities_char_zero}
    Given a quasi-excellent Noetherian scheme of equal characteristic zero, the following are equivalent:
    \begin{enumerate}
        \item $X$ having rational singularities
        \item $X$ being a derived splinter.
    \end{enumerate}
    This is \cite[Theorem 2.12]{Bhatt:2012}, \cite[Theorem 3]{Kovac:2000}, and \cite[Theorem 9.5]{Murayama:2024}.
\end{remark}

\begin{definition}
    A locally quasi-excellent Noetherian scheme $X$ of equal characteristic zero is said to have \textbf{semi-rational singularities} if for every resolution of singularities $\pi \colon \widetilde{X} \to X$, one has the natural map $\pi_\ast \mathcal{O}_{\widetilde{X}} \to \mathbb{R}\pi_\ast \mathcal{O}_{\widetilde{X}}$ is an isomorphism in $D^b_{\operatorname{coh}}(X)$. 
\end{definition}

\begin{remark}\label{rmk:du_bois_complex}
    Let $X$ be a variety over a field of characteristic zero. We can associate to $X$ a filtered complex, denoted $(\underline{\Omega}^\bullet_X, F)$, in $D(X)$ satisfying some useful properties, see \cite{Steenbrink:1981,DuBois:1981}. This filtered complex is often called the \textit{Du Bois complex}, and can be viewed as an analog of the de Rham complex for singular varieties. The graded pieces of this filtered complex, suitably shifted, are denoted $\underline{\Omega}^i_X$ and are objects in $D^b_{\operatorname{coh}}(X)$. See \cite[Theorem 3.1]{Schwede:2007} for further background.
\end{remark}

\begin{definition}
    A variety $X$ over a field of characteristic zero is said to have \textbf{Du Bois singularities} if the natural map $\mathcal{O}_X \to \underline{\Omega}^0_X$ is an isomorphism in $D^b_{\operatorname{coh}}(X)$.
\end{definition}

\begin{remark}\label{rmk:du_bois_properties}
    \hfill
    \begin{enumerate}
        \item Any variety over a field of characteristic zero with rational singularities must have Du Bois singularities, see \cite[Theorem S]{Kovacs:1999}. There are examples of such varieties with Du Bois singularities, but not rational singularities, see \cite{Chou/Song:2018}.
        \item Suppose $X$ is a variety over a field of characteristic zero such that there exists a closed immersion $i \colon X \to Y$ where $Y$ is a smooth variety (which is called an \textit{embeddable} variety). Let $\pi \colon \widetilde{Y} \to Y$ be a log resolution of the pair $(Y,X)$ that is an isomorphism away from $X$. Denote by $E$ the inverse image of $X$ along $\pi$ with the reduced induced closed subscheme structure in $\widetilde{Y}$. The following are noted from \cite[Theorem 4.3]{Schwede:2007} and \cite[Theorem 2.3]{Kovacs:1999}:
        \begin{enumerate}
            \item If $p \colon E \to X$ is the natural morphism, then there exists an isomorphism $\underline{\Omega}^0_X \to \mathbb{R}p_\ast \mathcal{O}_E$ in $D^b_{\operatorname{coh}}(X)$.
            \item $X$ has Du Bois singularities if, and only if, the natural map $\mathcal{O}_X \to \mathbb{R}p_\ast \mathcal{O}_E$ splits in $D^b_{\operatorname{coh}}(X)$.
        \end{enumerate}
    \end{enumerate}
\end{remark}

\subsection{Proofs}
\label{sec:proofs}

To start proving our results, we need to understand, for a given proper morphism of schemes $\pi \colon Y\to X$, how the natural map $\mathcal{O}_X \to \mathbb{R}\pi_\ast \mathcal{O}_Y$ splitting is related to the level of $\mathcal{O}_X$ with respect to the essential image of $\mathbb{R}\pi_\ast \colon D^b_{\operatorname{coh}}(Y) \to D^b_{\operatorname{coh}}(X)$. We state the following lemma which is an improvement to 
an earlier version of our work, based on input by Bhargav Bhatt.

\begin{lemma}\label{lem:splitting_via_thick_tricat}
    Consider a map $f\colon A \to B$ of algebras in a symmetric monoidal triangulated category $\mathcal{D}$. Suppose $M$ is a $B$-module in $\mathcal{D}$ such that $A$ is a direct summand of $f_\ast M$, with $M$ regarded as an $A$-module by restriction along $f$. Then the algebra map $A \to B$ splits in $\mathcal{D}$.
\end{lemma}

We will apply Lemma~\ref{lem:splitting_via_thick_tricat} to the special case where $\mathcal{D} =D_{\operatorname{Qcoh}}(X)$ and $A \to B$ is the natural map $\mathcal{O}_X \to \mathbb{R}\pi_\ast \mathcal{O}_Y$ for a proper morphism $\pi \colon Y \to X$ of Noetherian schemes. If we are given a pair of objects $E,G$ in $D^b_{\operatorname{coh}}(X)$ and $E$ is a direct summand of $G$ in $D_{\operatorname{Qcoh}}(X)$, then $E$ is also a direct summand of $G$ in $D^b_{\operatorname{coh}}(X)$ since $D^b_{\operatorname{coh}}(X)$ is a full subcategory.

\begin{proof}[Proof of Lemma~\ref{lem:splitting_via_thick_tricat}]
    It suffices to check the splitting after tensoring over $A$ with $f_\ast M$ (as $f_\ast M$ contains $A$ as a direct summand). But then the statement is clear: since $M$ is a $B$-module, the action map $B \otimes_A f_\ast M \to M$ provides the desired splitting of $M \to B \otimes_A f_\ast M$.
\end{proof}

\begin{theorem}\label{thm:derived_splinters_via_generation}
    If $X$ is a Noetherian scheme, then the following are equivalent:
    \begin{enumerate}
        \item $X$ is a derived splinter
        \item $\operatorname{level}^{\mathbb{R}\pi_\ast D^b_{\operatorname{coh}}(Y)} (\mathcal{O}_X) \leq 1$ for all proper surjective morphisms $\pi \colon Y \to X$.
    \end{enumerate}
    There is an analogous statement for splinters.
\end{theorem}

\begin{proof}
    We show the case for derived splinters. It is clear that $(1)\implies (2)$. Let $\pi \colon Y \to X$ be a proper surjective morphism. Our hypothesis tells us that $\mathcal{O}_X$ belongs to $\langle \mathbb{R} \pi_\ast D^b_{\operatorname{coh}}(Y) \rangle_1$. Hence, $\mathcal{O}_X$ belongs to $\langle \mathbb{R} \pi_\ast E \rangle_1$ for some $E$ in $D^b_{\operatorname{coh}}(Y)$. However, this ensures that natural map $\mathcal{O}_X \to \mathbb{R}\pi_\ast \mathcal{O}_Y$ splits in $D^b_{\operatorname{coh}}(X)$ via Lemma~\ref{lem:splitting_via_thick_tricat}. Since $\pi$ was arbitrary, we see that $X$ must be a derived splinter, and so $(2)\implies (1)$. The proof for splinters follows similarly.
\end{proof}

\begin{theorem}\label{thm:rational_singularities_via_generation}
    Suppose $X,Y$ are integral Noetherian schemes of characteristic zero. If $\pi \colon Y \to X$ is a proper surjective morphism where $Y$ has rational singularities, then the following are equivalent:
    \begin{enumerate}
        \item $X$ has rational singularities
        \item $\operatorname{level}^{\mathbb{R}\pi_\ast D^b_{\operatorname{coh}}(Y)} (\mathcal{O}_X)\leq 1$.
    \end{enumerate}
\end{theorem}

\begin{proof}
    It is evident that $(1)\implies (2)$, so we show that $(2)\implies (1)$. If $\mathcal{O}_X$ is in $\langle \mathbb{R}\pi_\ast D^b_{\operatorname{coh}}(Y) \rangle_1$ for proper surjective morphism $\pi \colon Y \to X$ where $Y$ has rational singularities, then the natural map $\mathcal{O}_X \to \mathbb{R}\pi_\ast \mathcal{O}_Y$ splits in $D^b_{\operatorname{coh}}(X)$, see Lemma~\ref{lem:splitting_via_thick_tricat}. However, as $Y$ has rational singularities, then so must $X$, see \cite[Theorem 9.2]{Murayama:2024}. This completes the proof.
\end{proof}

\begin{theorem}\label{thm:du_bois_via_generation}
    If $X$ is an embeddable variety over a field of characteristic zero, then the following are equivalent:
    \begin{enumerate}
        \item $X$ has Du Bois singularities
        \item $\operatorname{level}^{\underline{\Omega}^0_X} (\mathcal{O}_X) \leq 1$.
    \end{enumerate}
\end{theorem} 

\begin{proof}
    It is evident that $(1)\implies (2)$, so we check the converse. Since $X$ is embeddable, there exists a closed immersion $i \colon X \to Y$ such that $Y$ is a smooth variety over $k$. Consider a log resolution $\pi \colon \widetilde{Y}\to Y$ of the pair $(Y,X)$ that is an isomorphism outside of $X$. Let $E$ denote the inverse image of $X$ along $\pi$ with the reduced induced subscheme structure, and $p \colon E \to X$ the natural morphism. By Remark~\ref{rmk:du_bois_properties}, we have that $\mathbb{R}p_\ast \mathcal{O}_E$ is isomorphic to $\underline{\Omega}^0_X$ in $D^b_{\operatorname{coh}}(X)$. If $\mathcal{O}_X$ is in $\langle \mathbb{R}p_\ast \mathcal{O}_E \rangle_1$, Lemma~\ref{lem:splitting_via_thick_tricat} tells us the natural map $\mathcal{O}_X \to \mathbb{R}p_\ast \mathcal{O}_E$ splits. Hence, Remark~\ref{rmk:du_bois_properties} tells us $X$ has Du Bois singularities.
\end{proof}

\begin{remark}
In the statements above, the special case where $X$ is additionally \textit{proper} is interesting. We can appeal to Lemma~\ref{lem:projective_over_complete_local_ring_krs_category} to check that $D^b_{\operatorname{coh}}(X)$ is a Krull-Schmidt category, and hence, observe the following with the respective notations as earlier:
\begin{enumerate}
    \item $X$ is a derived splinter if, and only if, $\operatorname{level}^{\mathcal{O}_X} (\mathbb{R}\pi_\ast E) \leq 1$ for some object $E$ in $D^b_{\operatorname{coh}}(Y)$
    \item $X$ has rational singularities if, and only if, $\operatorname{level}^{\mathcal{O}_X} (\mathbb{R}\pi_\ast E) \leq 1$ for some object $E$ in $D^b_{\operatorname{coh}}(Y)$
    \item $X$ has Du Bois singularities if, and only if, $\operatorname{level}^{\mathcal{O}_X} (\underline{\Omega}^0_X) \leq 1$ .
\end{enumerate}
\end{remark}

\begin{proof}[Proof of Theorem~\ref{introthm:normality_via_generation}]
    Before we begin the proof, observe that since $\nu$ is a finite morphism, we have $\nu_\ast \mathcal{O}_{X^\nu} = \mathbb{R} \nu_\ast \mathcal{O}_{X^\nu}$. If $X$ is normal, then $\mathcal{O}_X \simeq \nu_\ast \mathcal{O}_{X^\nu}$ and so $(1)\implies (2)$. It suffices to show $(2)\implies (1)$.
    
    Suppose $\mathcal{O}_X$ is in $\langle \nu_\ast \mathcal{O}_{X^\nu} \rangle_1$ in $D^b_{\operatorname{coh}}(X)$. If $X$ is a reduced Nagata scheme, then the normalization $\nu \colon X^\nu \to X$ is a finite birational morphism, see \cite[\href{https://stacks.math.columbia.edu/tag/035S}{Tag 035S}]{StacksProject}. By Lemma~\ref{lem:splitting_via_thick_tricat}, the natural map $\mathcal{O}_X \to \nu_\ast \mathcal{O}_{X^\nu}$ splits in $D^b_{\operatorname{coh}} (X)$. Hence, there exists a split short exact sequence in $\operatorname{coh}X$:
    \begin{displaymath}
        0 \to \mathcal{O}_X \to \nu_\ast \mathcal{O}_{X^\nu} \to C \to 0.
    \end{displaymath}
    Let $\eta$ be the generic point of $X$. Note that $\nu_\ast \mathcal{O}_{X^\nu}$ is a torsion free sheaf of rank one, see \cite[Proposition 7.4.5]{EGAI/II/III:IHES}. Passing to the generic point, we see that $C_\eta$ must vanish, and so $C$ is a torsion sheaf on $X$. However, $C$ is a direct summand of the torsion free sheaf $\nu_\ast \mathcal{O}_{X^\nu}$, so $C=0$. This tells us that the natural map $\mathcal{O}_X \to \nu_\ast \mathcal{O}_{X^\nu}$ is an isomorphism. Thus, every affine open $\operatorname{Spec}(R)$ of $X$ is normal since $R \to R^\nu$ is an isomorphism, and so we get that $X$ is normal.
\end{proof}

\begin{lemma}\label{lem:bhatt_lemma}(Bhatt)
    Suppose $f\colon Y \to X$ is a proper birational map of Noetherian schemes such that $E:=\mathbb{R}\pi_\ast \mathcal{O}_Y$ has finite Tor dimension (or equivalently, is a perfect complex). Then the natural map $\mathcal{O}_X \to \mathbb{R}\pi_\ast \mathcal{O}_Y$ splits.
\end{lemma}

\begin{proof}
    Since $E$ is a perfect complex, there is a trace map $\mathrm{tr} \colon \mathbb{R}\mathcal{H}\mathrm{om}_X(E,E) \to \mathcal{O}_X$. Since $E$ is an algebra, there is an action map $E \to \mathbb{R}\mathcal{H}\mathrm{om}_X(E,E)$. We can then consider the composition of natural maps:
    \begin{displaymath}  
        \mathcal{O}_X \to E \to \mathbb{R}\mathcal{H}\mathrm{om}_X(E,E) \xrightarrow{\operatorname{tr}} \mathcal{O}_X.
    \end{displaymath}
    This composition is the identity since it is an identity generically on $X$ by the birationality assumption.
\end{proof}

\begin{proposition}\label{prop:rational_singularities_perfect}
    Suppose $X,Y$ are integral Noetherian schemes of characteristic zero. If $\pi \colon Y \to X$ is a proper birational morphism where $Y$ has rational singularities, then the following are equivalent:
    \begin{enumerate}
        \item $X$ has rational singularities
        \item $\operatorname{level}^{\mathbb{R}\pi_\ast D^b_{\operatorname{coh}}(Y)} (\mathcal{O}_X) \leq 1$
        \item $\mathbb{R}\pi_\ast \mathcal{O}_Y$ is in $\operatorname{perf}X$.
    \end{enumerate}
\end{proposition}

\begin{proof}
    By Theorem~\ref{thm:rational_singularities_via_generation}, we know that $(1) \iff (2)$. The definition of rational singularities promises us $(1)\implies (3)$. If $(3)$ holds, then Lemma~\ref{lem:bhatt_lemma} ensures that the natural map $\mathcal{O}_X \to \mathbb{R}\pi_\ast \mathcal{O}_Y$ splits in $D^b_{\operatorname{coh}}(X)$. From this, it follows that $X$ has rational singularities, see \cite[Theorem 9.2]{Murayama:2024}.
\end{proof}

\section{Examples}
\label{sec:examples}

This section establishes explicit computations yielding bounds on the number of cones needed to finitely build $\mathcal{O}_X$ in $D^b_{\operatorname{coh}}(X)$ from a given resolution of singularities $\pi \colon \widetilde{X} \to X$ for a variety $X$ over a field. The examples of interest include one-dimensional Nagata schemes (i.e. curves) and affine cones over a smooth projective hypersurface. 

\subsection{One-dimensional Nagata schemes}
\label{sec:one_dimensional_schemes}

Let us recall some useful properties for one-dimensional Nagata schemes.

\begin{remark}
    Let $X$ be a reduced Nagata one-dimensional scheme and $p$ be a closed point of $X$.
    \begin{enumerate}
        \item The normalization $\nu \colon X^\nu \to X$ is a finite birational morphism, see \cite[\href{https://stacks.math.columbia.edu/tag/0C1R}{Tag OC1R}]{StacksProject}. 
        \item The \textbf{$\delta$-invariant} of $\mathcal{O}_{X,p}$, denoted $\delta_p$, is defined as the length of $\mathcal{A}_p/\mathcal{O}_{X,p}$ as an $\mathcal{O}_{X,p}$-module where $\mathcal{A}_p$ is the integral closure of $\mathcal{O}_{X,p}$ in the total ring of fractions of $\mathcal{O}_{X,p}$, see \cite[\href{https://stacks.math.columbia.edu/tag/0C3Q}{Tag 0C3Q}]{StacksProject} for details. \item The $\delta$-invariant of $\mathcal{O}_{X,p}$ is bounded below by one less than number of its geometric branches \cite[\href{https://stacks.math.columbia.edu/tag/0C43}{Tag 0C43}]{StacksProject}. Equivalently, it is bounded below by the following (see \cite[\href{https://stacks.math.columbia.edu/tag/0C1S}{Tag 0C1S}]{StacksProject}):
        \begin{displaymath}
            \sum_{\nu(q)=p} |\kappa (q): \kappa(p)|_s
        \end{displaymath}
        where $\nu \colon X^\nu \to X$ is the normalization of $X$ and the sum ranges over all $q$ in $X^\nu$ such that $\nu(q)=p$.
    \end{enumerate}
\end{remark}

\begin{proposition}\label{prop:rouquier_dimension_curves}
    Let $X$ be a reduced Nagata one-dimensional scheme. If $\nu \colon X^\nu \to X$ is the normalization of $X$, then the Rouquier dimension of $D^b_{\operatorname{coh}}(X)$ is bounded above by the following value:
    \begin{displaymath}
        (\dim D^b_{\operatorname{coh}}(X^\nu) +1 )(1 + \underset{s \in \operatorname{Sing}(X)}{\max} \{ \delta_s\}) -1
    \end{displaymath}
    where $\delta_s$ denotes the $\delta$-invariant of $X$ at $s$.
\end{proposition}

\begin{proof}
    Let $G$ be a compact generator for $D_{\operatorname{Qcoh}}(X^\nu)$, and assume that $\mathcal{O}_{X^\nu}$ is a direct summand of $G$. Throughout this proof, we will freely utilize various part of \cite[\href{https://stacks.math.columbia.edu/tag/0C1R}{Tag OC1R}]{StacksProject}. Note that $\nu$ is a finite birational morphism. There exists a short exact sequence in $\operatorname{coh}X$:
    \begin{displaymath}
        0 \to \mathcal{O}_X \to \nu_\ast \mathcal{O}_{X_\nu} \to \bigoplus_{s\in \operatorname{Sing}(X)} \mathcal{Q}_s \to 0
    \end{displaymath}
    where $\mathcal{Q}_s$ is a skyscraper sheaf supported on the singular point $s$ in $X$. Note that each $\mathcal{Q}_s$, for $s$ in $\operatorname{Sing}(X)$, has length equal to the $\delta$-invariant of $X$ at $s$, which we denote by $\delta_s$. For each $s$ in $\operatorname{Sing}(X)$, consider the fiber square:
    \begin{displaymath}
        \begin{tikzcd}
            Z_s \ar[r, "p"] \ar[d, "q"] & \operatorname{Spec}(\kappa(s)) \ar[d, "i"]\\
            X^\nu \ar[r, "\nu"] & X.
        \end{tikzcd}
    \end{displaymath}
    First, since the maps are finite, we have that $\mathcal{O}_{\operatorname{Spec}{\kappa(s)}} \to p_\ast \mathcal{O}_{Z_s}$ splits, and so $\mathcal{O}_{\operatorname{Spec}{\kappa(s)}}$ belongs to $\langle p_\ast \mathcal{O}_{Z_s} \rangle_1$. This implies that $i_\ast \mathcal{O}_{\operatorname{Spec}{\kappa(s)}}$ belongs to $\langle (i \circ p)_\ast\mathcal{O}_{Z_s} \rangle_1$, and since $(i \circ p)_\ast = (\nu \circ q)_\ast$, we have $i_\ast \mathcal{O}_{\operatorname{Spec}{\kappa(s)}}$ belongs to $\langle (\nu \circ q)_\ast \mathcal{O}_{Z_s} \rangle_1$ and hence, $ \operatorname{level}^{(\nu \circ q)_\ast \mathcal{O}_{Z_s}}(i_\ast \mathcal{O}_s) \leq  1$. Since $X$ has finitely many singular points, we may assume $G$ finitely builds each $q_\ast \mathcal{O}_{Z_s}$ in at most one step by simply adding to $G$ a copy of such object. We have $\operatorname{level}^{\nu_\ast G}(i_\ast \mathcal{O}_{\operatorname{Spec}{\kappa(s)}}) \leq 1$, and so $\operatorname{level}^{\nu_\ast G}(Q_s) \leq \operatorname{level}^{\nu_\ast G}(i_\ast \mathcal{O}_{\operatorname{Spec}{\kappa(s)}}) \cdot \operatorname{level}^{i_\ast \mathcal{O}_{\operatorname{Spec}{\kappa(s)}}}(Q_s)\leq \delta_s$ as $Q_s$ is of length $\delta_s$. Now, we finish the proof by observing that:
    \begin{displaymath}
        \begin{aligned}
            \operatorname{level}^{\nu_\ast G} (\mathcal{O}_X) &\leq \operatorname{level}^{\nu_\ast G}(\nu_\ast \mathcal{O}_{X^\nu}) + \operatorname{level}^{\nu_\ast \mathcal{O}_{X^\nu}}(\bigoplus_{s \in \operatorname{Sing}(X)} Q_s)\\
            &= 1 + \max_{s \in \operatorname{Sing}(X)} \{ \operatorname{level}^{\nu_\ast \mathcal{O}_{X^\nu}}(Q_s) \}\\
            &\leq 1 + \max_{s \in \operatorname{Sing}(X)} \{ \delta_s\}
        \end{aligned}
    \end{displaymath}
    The claim follows from \cite[Proposition 3.16]{Lank/Olander:2024}.
\end{proof}

\begin{corollary}\label{cor:affine_curve_level}
    Let $X$ be a reduced Nagata one-dimensional affine scheme. If $\nu \colon X^\nu \to X$ is the normalization of $X$, then the level of $\mathcal{O}_X$ with respect to $\nu_\ast \mathcal{O}_{X^\nu}$ is bounded above by the following value:
    \begin{displaymath}
        1 + \underset{s \in \operatorname{Sing}(X)}{\max} \{ \delta_s\}
    \end{displaymath}
    where $\delta_s$ denotes the $\delta$-invariant of $X$ at $s$.
\end{corollary}

\begin{proof}
    This is essentially the proof of Proposition~\ref{prop:rouquier_dimension_curves} coupled with the fact that $\mathcal{O}_{X^\nu}$ is a strong generator for $D^b_{\operatorname{coh}}(X^\nu)$ with generation time one.
\end{proof}

\subsection{Cones over smooth projective varieties}
\label{sec:affine_projective_cones}

We fix an integer $n$ which is at least two. Let $X$ be a smooth projective variety over a field $k$ of characteristic zero. Let $\mathcal{L}$ be an ample line bundle on $X$. The \textit{affine cone} over $X$ with conormal bundle $\mathcal{L}$ is defined as the affine variety:
\begin{displaymath}
    C:= \operatorname{Spec}(\bigoplus_{m\geq 0} H^0 (X, \mathcal{L}^{\otimes m})).
\end{displaymath}
Since $\mathcal{L}$ is ample, the ring is $\bigoplus_{m\geq 0} H^0 (X, \mathcal{L}^{\otimes m})$ is finitely generated. Note that $C$ is a normal variety of Krull dimension $\dim X +1$ whose regular locus is all of $C$ except for a single point $v$, called the \textit{vertex}. We denote by $\mathfrak{m}$ the ideal sheaf defining $v$. For background, see \cite[Section 3.1]{Kollar:2013}.

\begin{setup}\label{setup:cones}
    Consider the special case where $t\colon X \to \mathbb{P}^n$ is a smooth hypersurface of degree $d$. Let $\mathcal{L} := t^\ast \mathcal{O}_{\mathbb{P}^n}(1)$. We can resolve the singularities of $C$ by simply blowing up $C$ at the vertex $v$, as summarized in the following diagram:
    \begin{displaymath}
        \begin{tikzcd}
            E \ar[r, hook, "i"] \ar[d, "p"] & \widetilde{C} \ar[d, "\pi"]\\
            v \ar[r, hook, "j"] & C.
        \end{tikzcd}
    \end{displaymath}
    Here, $\pi$ denotes the blowup of $C$ at $v$, and $E = (\pi^{-1}(v))_{\operatorname{red}}$ denotes the reduced exceptional divisor. It is important to note that $E$ is isomorphic to $X$.
\end{setup}

\begin{lemma}\label{lem:cone_annihilator_bound}
    Assume Setup~\ref{setup:cones}. If $d \geq n$, then for every integer $c >0$, we have that $\mathfrak{m}^{d-n}$ annihilates $\mathbb{R}^c \pi_\ast \mathcal{O}_{\widetilde{C}}$.
\end{lemma}

\begin{proof}
    Let $\omega_C$ denote the canonical bundle on $C$. Since $\pi$ is an isomorphism away from $E$, we can write the canonical bundle on $\widetilde{C}$ as $\omega_{\widetilde{C}} = \pi^\ast\omega_C\otimes \mathcal{O}_{\widetilde{C}}(aE)$ for some integer $a$. Let us first calculate $a$ in terms of $n$ and $d$. The adjunction formula for the hypersurface $E$ contained in $\widetilde{C}$ gives us:
    \begin{displaymath}
        \begin{aligned}
            \omega_E &= i^\ast ( \omega_{\widetilde{C}} \otimes \mathcal{O}_{\widetilde{C}}(E) ) \\
            &= i^\ast (\pi^\ast\omega_C\otimes \mathcal{O}_{\widetilde{C}}(aE) \otimes \mathcal{O}_{\widetilde{C}}(E) )\\
            &= i^\ast (\pi^\ast\omega_C\otimes \mathcal{O}_{\widetilde{C}}((a+1)E) )\\
            &= i^\ast \mathcal{O}_{\widetilde{C}}((a+1)E),
        \end{aligned}
    \end{displaymath}
    where for the final equality, we use the fact that $\pi^\ast\omega_C$ is trivial along the fiber $E$. Now, under the identification $E$ being isomorphic to $X$, we make two observations. First, we have the following isomorphism:
    \begin{displaymath}
        \begin{aligned}
            \omega_E &\cong \omega_X \\&\cong t^\ast ( \omega_{\mathbb{P}^n} \otimes \mathcal{O}_{\mathbb{P}^n}(d)) \\&= t^\ast \mathcal{O}_{\mathbb{P}^n}(-n-1+d).
        \end{aligned}
    \end{displaymath}
    Secondly, $i^\ast \mathcal{O}_{\widetilde{C}}((a+1)E) \cong t^\ast \mathcal{O}_{\mathbb{P}^n}(-(a+1))$. If we compare the two quantities, then we have that $-n-1+d = -a-1$, and so $a = n-d$. By Grauert-Riemenschneider vanishing, we see that for any positive integer $c$:
    \begin{displaymath}
        \begin{aligned}
        0 &= \mathbb{R}^c\pi_\ast\omega_{\widetilde{C}} \\&= \mathbb{R}^c\pi_\ast(\pi^\ast\omega_{C} \otimes \mathcal{O}_{\widetilde{C}}((n-d)E))\\
        &= \omega_C \otimes \mathbb{R}^c\pi_\ast\mathcal{O}_{\widetilde{C}}((n-d)E),
        \end{aligned}
    \end{displaymath}
    where the last equality follows from the projection formula. Thus, one has $\mathbb{R}^c\pi_\ast\mathcal{O}_{\widetilde{C}}((n-d)E) = 0$ for $c>0$. Now, given any integer $l$, consider the short exact sequence on $\widetilde{C}$:
    \begin{displaymath} 
        0 \to \mathcal{O}_{\widetilde{C}}(lE) \to \mathcal{O}_{\widetilde{C}}((l+1)E) \to i_\ast\mathcal{O}_E((l+1)E) \to 0 
    \end{displaymath}
    Pushing this forward by by $\pi_\ast$, we get a long exact sequence in cohomology sheaves:
    \begin{displaymath} 
        \dots \to \mathbb{R}^c\pi_\ast\mathcal{O}_{\widetilde{C}}(lE) \to \mathbb{R}^c\pi_\ast\mathcal{O}_{\widetilde{C}}((l+1)E) \to \mathbb{R}^c\pi_\ast i_\ast\mathcal{O}_E((l+1)E) \to \dots 
    \end{displaymath}
    For every $c>0$, observe that $\mathbb{R}^c\pi_\ast i_\ast \mathcal{O}_E((l+1)E) = j_\ast \mathbb{R}^c p_\ast\mathcal{O}_E((l+1)E)$ and so, it is annihilated by $\mathfrak{m}$ since it is scheme-theoretically supported on $v$. It follows that if $\mathbb{R}^c\pi_\ast\mathcal{O}_{\widetilde{C}}(lE)$ is annihilated by $\mathfrak{m}^k$ for some $k$, then the middle term $\mathbb{R}^c\pi_\ast\mathcal{O}_{\widetilde{C}}((l+1)E)$ is annihilated by $\mathfrak{m}^{k+1}$. Applying this argument step by step, we have that:
    \begin{displaymath}
        \begin{aligned}
            \mathfrak{m}^0 \text{ annihilates } \mathbb{R}^c\pi_\ast\mathcal{O}_{\widetilde{C}}((n-d)E) \implies &\mathfrak{m}^1 \text{ annihilates } \mathbb{R}^c\pi_\ast\mathcal{O}_{\widetilde{C}}((n-d+1)E)\\
            &\vdots\\
            \implies &\mathfrak{m}^k \text{ annihilates } \mathbb{R}^c\pi_\ast\mathcal{O}_{\widetilde{C}}((n-d+k)E)\\
            &\vdots\\
            \implies &\mathfrak{m}^{d-n} \text{ annihilates } \mathbb{R}^c\pi_\ast\mathcal{O}_{\widetilde{C}} \text{ for all $c>0$}.
        \end{aligned}
    \end{displaymath}
\end{proof}

In Theorem~\ref{introthm:rational_derived_splinter_db_generation}, we saw a relation between the level of $\mathcal{O}_{\tilde C}$ with respect to $\mathbb{R}\pi_\ast\mathcal{O}_{\widetilde{C}}$ and $C$ having rational singularities. Since $C$ has rational singularities if and only if $d \leq n$, we state the following proposition describing the level of $\mathcal{O}_{\tilde C}$, albeit with respect to $\mathbb{R}\pi_\ast(\mathcal{O}_{\widetilde{C}} \oplus \mathcal{O}_{\widetilde{C}}(E))$, in terms of $d$ and $n$.

\begin{proposition}\label{prop:cones_inequality}
    Assume Setup~\ref{setup:cones}. If $d \geq n$, then there is an inequality:
    \begin{displaymath}
    \operatorname{level}^{\mathbb{R}\pi_\ast(\mathcal{O}_{\widetilde{C}} \oplus \mathcal{O}_{\widetilde{C}}(E))}(\mathcal{O}_C) \leq 1+ 2(d-n).
    \end{displaymath}
\end{proposition}

\begin{proof}
    Consider the diagram in Setup~\ref{setup:cones}:
    \begin{displaymath}
        \begin{tikzcd}
            E \ar[r, hook, "i"] \ar[d, "p"] & \widetilde{C} \ar[d, "\pi"]\\
            v \ar[r, hook, "j"] & C.
        \end{tikzcd}
    \end{displaymath}
    Observe that under the identification $E$ being isomorphic to $X$, we have $\mathcal{O}_E(E)$ is isomorphic to $t^\ast \mathcal{O}_{\mathbb{P}^n}(-1)$. This gives us a short exact sequence:
    \begin{displaymath} 
        0 \to \mathcal{O}_{\mathbb{P}^n}(-d-1) \to \mathcal{O}_{\mathbb{P}^n}(-1) \to t_\ast t^\ast \mathcal{O}_{\mathbb{P}^n}(-1) \to 0 
    \end{displaymath}
    From the long exact sequence in cohomology, $d\geq n$ implies:
    \begin{displaymath} 
        H^{n-1}(X, t^\ast \mathcal{O}_{\mathbb{P}^n}(-1)) = H^n(\mathbb{P}^n,\mathcal{O}_{\mathbb{P}^n}(-d-1)) \neq 0 . 
    \end{displaymath}
    In particular, this implies that $\mathbb{R} p_\ast\mathcal{O}_E(E) \neq 0$ since for every $c$:
    \begin{displaymath}
        \mathbb{R}^c p_\ast\mathcal{O}_E(E) = H^c(E,\mathcal{O}_E(E)) = H^c(X,t^\ast \mathcal{O}_{\mathbb{P}^n}(-1)).
    \end{displaymath}
    As a consequence, we get that $\mathcal{O}_{\operatorname{Spec}(\kappa(v))}$ belongs to $\langle\mathbb{R}p_\ast\mathcal{O}_E(E) \rangle_1$. Therefore, $j_\ast \mathcal{O}_{\operatorname{Spec}(\kappa(v))}$ belongs to $\langle i_\ast \mathbb{R}p_\ast\mathcal{O}_E(E) \rangle_1$, and hence, to $\langle\mathbb{R}\pi_\ast i_\ast \mathcal{O}_E(E) \rangle_1$. Additionally, we have that $\mathbb{R}\pi_\ast i_\ast \mathcal{O}_E(E)$ is in $\langle \mathbb{R}\pi_\ast(\mathcal{O}_{\widetilde{C}} \oplus \mathcal{O}_{\widetilde{C}}(E)) \rangle_2$ since we have the short exact sequence:
    \begin{displaymath} 
        0 \to \mathcal{O}_{\widetilde{C}} \to \mathcal{O}_{\widetilde{C}}(E) \to i_\ast \mathcal{O}_E(E) \to 0 
    \end{displaymath}
    We have the distinguished triangle:
    \begin{displaymath} 
        \mathcal{O}_C \xrightarrow{\operatorname{ntrl.}} \mathbb{R}\pi_\ast\mathcal{O}_{\widetilde{C}} \to Q \to \mathcal{O}_C[1]. 
    \end{displaymath}
    Observe that $Q$ is supported in cohomological degrees $c>0$, and the cohomology sheaves $\mathcal{H}^c (Q) = \mathbb{R}^c\pi_\ast\mathcal{O}_{\widetilde{C}}$ for $c>0$. By Lemma \ref{lem:cone_annihilator_bound}, we have that $\mathfrak{m}^{d-n}$ annihilates $\mathcal{H}^c (Q)$ for all $c$. In particular, we see that $Q$ is scheme-theoretically supported on the $(d-n)$-th nilpotent thickening of the ideal sheaf associated to the vertex. This promises that $j_\ast \mathcal{O}_{\operatorname{Spec}(\kappa(v))}$ finitely builds $Q$ in at most $d-n$ cones. The desired claim follows immediately.
\end{proof}

\begin{corollary}\label{cor:cones_rouq_dim}
    Assume Setup~\ref{setup:cones}. If $d \geq n$, then there is an inequality:
    \begin{displaymath}
        \dim D^b_{\operatorname{coh}}(C) \leq (2 \dim C +1 )(1+ 2(d-n)) -1.
    \end{displaymath}
\end{corollary}

\begin{proof}
    Let $\mathcal{L}$ be an ample line bundle on $\widetilde{C}$. Then $G:=\bigoplus^{\dim C}_{i=0} \mathcal{L}^{\otimes i}$ is a strong generator for $D^b_{\operatorname{coh}}(\widetilde{C})$ whose generation time is at most $2 \dim \widetilde{C}$, see Example~\ref{ex:strong_generator_examples}. By Proposition~\ref{prop:cones_inequality}, it follows that $\mathcal{O}_C$ belongs to $\langle \mathbb{R}\pi_\ast (G \oplus \mathcal{O}_{\widetilde{C}} (E)) \rangle_{1 + 2(d-n)}$. By \cite[Proposition 3.16]{Lank/Olander:2024}, the desired bound follows.
\end{proof}

\begin{remark}
    It is a conjecture of Orlov that the Rouquier dimension of $D^b_{\operatorname{coh}}(W)$, for $W$ a smooth variety over a field, coincides with Krull dimension, see \cite[Conjecture 10]{Orlov:2009}. If this conjecture holds, then there is an obvious sharper bound similar to Corollary~\ref{cor:cones_rouq_dim}.
\end{remark}

\bibliographystyle{alpha}
\bibliography{mainbib}

\end{document}